\documentclass[10pt]{amsart}
\usepackage{amsmath,amssymb,amsfonts,amscd}
\usepackage[english]{babel}
\usepackage[utf8]{inputenc}
\usepackage{graphicx}
\usepackage[colorinlistoftodos]{todonotes}
\usepackage{parskip}
\usepackage{listings}
\usepackage{multicol}
\usepackage{pgf}
\usepackage{tikz}
\usepackage{array}
\usepackage{fancyhdr}
\usetikzlibrary{arrows}
\usetikzlibrary {positioning}
\definecolor {processblue}{cmyk}{0.96,0,0,0}

\newtheorem{theorem}{Theorem}[section]
\newtheorem{proposition}[theorem]{Proposition}
\newtheorem{lemma}[theorem]{Lemma}

\theoremstyle{definition}
\newtheorem{definition}[theorem]{Definition}
\theoremstyle{remark}

\theoremstyle{remark}
\newtheorem{example}[theorem]{Example}
\theoremstyle{remark}

\begin{document}

\title{Toeplitz Subshifts with Trivial Centralizers and Positive Entropy}

\author{Kostya Medynets}
\email{medynets@usna.edu}
\address{Mathematics Department, United States Naval Academy, Annapolis, MD 21402}
\thanks{The research of K.M. was supported by NSA YIG H98230258656.}

\author{MIDN James P. Talisse}
\email{jptalisse@gmail.com}
\address{Mathematics Department, United States Naval Academy, Annapolis, MD 21402}


\begin{abstract}
Given a dynamical system $(X,G)$, the centralizer $C(G)$ denotes the group of all homeomorphisms of $X$ which commute with the action of $G$. This group is sometimes called the automorphism group of the dynamical system $(X,G)$. In this note, we generalize the construction of Bulatek and Kwiatkowski (1992) to $\mathbb Z^d$-Toepltiz systems by identifying a class of $\mathbb Z^d$-Toeplitz systems that have trivial centralizers. We show that this class of $\mathbb Z^d$-Toeplitz with trivial centralizers contains systems with positive topological entropy.

\end{abstract}
\maketitle

\section{Introduction}

Toeplitz dynamical systems were first introduced by Jacobs and Keane in \cite{jacobs19690}. They provided a classical definition for a Toeplitz sequence over $\{0, 1\}$. In \cite{markley1975substitution}, Markley studied these sequences and showed the equivalence of various definitions of them. The orbit closure of a Toeplitz sequence is regarded as a Toeplitz flow.
In \cite{markley1979almost}, Markley and Paul showed that these flows were exactly almost one-to-one extensions of odometers, or the group of $p$-adic integers. See \cite{hewitt2012abstract} for a general discussion of the group theoretic properties of the group of $p$-adic integers. For a general survey of symbolic dynamics, we refer the reader to Kitchens in \cite{kitchens2012symbolic}. For a good survey on $\mathbb{Z}$ odometers and Toeplitz flows, the reader is referred to \cite{downarowicz2005survey}. Recently the definition of Toeplitz flows was extended to flows over $\mathbb{Z}^d$ by Cortez in \cite{cortez2006z}, and then to flows over general groups in \cite{cortez2008g}, and by Krieger in \cite{krieger2010toeplitz}.

The centralizer of a dynamical system is the group of all homeomorphisms of the system which commute with the group action. Sometimes called the \textit{automorphism group} of the dynamical system in the literature, the centralizer of a dynamical system has an intricate relationship with its parent dynamical system. For example, in \cite{boyle1988automorphism}, Boyle, Lind and Rudolph study the centralizer of shifts of finite type and show that they are countable, residually finite and contain the free group on two generators. Several results have been shown by Cyr and Kra (\cite{cyr2015automorphism}, \cite{cyrkra2015automorphism}, \cite{cyr2016automorphism}) which relate varying levels of complexity of symbolic dynamical systems to algebraic properties of their centralizers. We notice that systems with positive entropy tend to have very large centralizers. For example, the centralizer of the full shift contains every finite group and the free group on two generators. On the other hand, Donoso, Durand and Petite showed that some classes of low complexity symbolic dynamical systems have very small centralizers, in the sense that they consist only of powers of $T$ in \cite{donoso2016automorphism}. Bulatek and Kwiatkowski study the centralizer of a class of high complexity Toeplitz systems in \cite{bulatek1990topological} and \cite{bulatek1992strictly}. The centralizer of multidimensional symbolic dynamical systems is studied by Hochman in \cite{hochman2010automorphism}. For example, he shows that the centralizer of a positive entropy multidimensional shift of finite type contains a copy of every finite group.

The main question this paper seeks to answer is whether there are multidimensional systems with a trivial centralizer and positive entropy. Following the ideas of Bulatek and Kwiatkowski in \cite{bulatek1992strictly}, which developed this result in one dimension, we establish this result with a constructive proof. We note that there are several constructions of {\it one-dimensional} Toeplitz systems with trivial centralizers and positive entropy, see, for example, \cite{donoso2017automorphismToeplitz} and references therein.

In Section \ref{Section:Odometer} we present main facts with proofs regarding general $G$-odometers, where $G$ is a residually finite group.
For the reader's convenience, we include the proofs, otherwise scattered across multiple sources. In particular we show that the centralizer group of $\mathbb Z^d$-Toeplitz systems embeds into the centralizer group of its maximal equicontinuous factor, which is a $\mathbb Z^d$-odometer, and so is abelian. This result was originally established by Auslander in \cite[Theorem 9]{auslander1963} using the techniques of enveloping semigroups. The proof we present in this note follows the approach developed in  \cite{olli2013endomorphisms}.

In section \ref{symbolicdynamics}, we construct a class of $\mathbb{Z}^d$-Toeplitz systems that have trivial centralizers. Then in Section \ref{section:positive entropy}, we show that this class contains systems of positive entropy. In this section we provide an explicit construction of a two dimensional Toeplitz of positive entropy.

\section{Definitions and Background}

By a {\it dynamical system} we mean a pair $(X,G)$, where $X$ is a compact topological space and $G$ is a countable discrete group acting on $X$ by homeomorphisms. The action of a group element $g\in G$ on $x\in X$ will be denoted by $g\cdot x = g(x)$.  The set $\{g \cdot x: g \in G\}$ is called the orbit of the point $x$. If every orbit of $(X,G)$ is dense, we call the system \textit{minimal}. A system $(X, G)$ is \textit{equicontinuous} if for all $\epsilon > 0$, there exists $\delta > 0$, such that for all $x, y \in X$ if $d(x, y) < \epsilon$, then $d(g \cdot x, g \cdot y) < \delta$ for all $g \in G$. Let $(X, G)$, and $(Y,G)$ be two minimal systems. If there exists a continuous surjection $\pi: X \to Y$ which preserves the action of $G$, we say that $X$ is an \textit{extension} of $Y$, and that $Y$ is a \textit{factor} of $X$. We call $\pi$ a \textit{factor map}. Given two factor maps $\pi$ and $\pi'$, we say that $\pi$ is \textit{larger} than $\pi'$ if there exists a third factor map $\pi''$ such that $\pi' = \pi'' \circ \pi$. As such, we can discuss the \textit{maximal} factor of a system. It is a known fact that every dynamical system has a maximal equicontinuous factor.

In this paper we are interested in symbolic dynamical systems. We start with a finite set $\Sigma$ called the alphabet. Say $|\Sigma| = n$. The set of all bi-infinite sequences over $\Sigma$ is called the full $n$-shift and is denoted $\Sigma^{\mathbb{Z}}$. In general, we denote the full $d$ dimensional $n$-shift by $\Sigma^{\mathbb{Z}^d}$. This set is endowed with the product topology from the discrete topology in each coordinate. Cylinder sets in which we fix a finite number of coordinates form a basis for the topology. For $x \in \Sigma^{\mathbb{Z}^d}$ we write $x = \{x(v)\}_{v \in \mathbb{Z}^d}$. We call $x$ a $\mathbb{Z}^d$ array. The group $\mathbb{Z}^d$ acts on $\Sigma^{\mathbb{Z}^d}$, denoted by $T^z(x)$ for $z \in \mathbb{Z}^d$ and $x \in \Sigma^{\mathbb{Z}^d}$ as follows: $T^z(x) = \{x(z + v)\}_{v \in \mathbb{Z}^d}$. The orbit of an array is $\{T^v(x): v \in \mathbb{Z}^d\}$. A subset $X \subseteq  \Sigma^{\mathbb{Z}^d}$ is called a \textit{subshift} if it is closed under the action of $\mathbb{Z}^d$.

For the sake of completeness, we note that symbolic dynamics can be studied over general, discrete groups. In this case, let $G$ be a discrete group. Then $\Sigma^G$ is acted on by the group $G$. While in this paper we restrict our study of symbolic dynamics to $\mathbb{Z}^d$ systems, we note that many of the results can be extended to $G$ systems for more general groups $G$.

The topological spaces discussed in this note will be topological zero-dimensional compact metric space without isolated points, i.e. a Cantor set. Notice that by a theorem of Brouwer (\cite{brouwer1910}) every Cantor set is homeomorphic to the middle-thirds Cantor set, and so all Cantor sets are homeomorphic.

\section{Odometers}\label{Section:Odometer}

In this section, we will recall some basic facts about odometers and their almost $1-1$ extensions. In particular, we show that the centralizer of an odometer is abelian, and the centralizer of the almost $1-1$ extension of an odometer is also abelian. These results are mostly known, but are scattered. In particular, the proof of Lemma \ref{proximal} appears in \cite{veech1970point} and the proof of Proposition \ref{onetoone} appears in \cite{olli2013endomorphisms}. We present slightly modified proofs for clarity and the reader's convenience.

\begin{definition}
A group $G$ is called \textit{residually finite} if the intersection of all its finite index normal subgroups is normal.
\end{definition}

\begin{definition}
Let $G$ be a residually finite group, and $G = G_0 \supseteq G_1 \supseteq G_2 \supseteq \ldots$ be nested normal subgroups such that $\bigcap G_n = \{0\}$. Let $\pi_n$ be the natural homomorphism from $G/G_n$ onto $G/G_{n-1}$, i.e. $\pi_n(hG_n) = hG_{n-1}$ for $h \in G$. The \textit{G-odometer}, $\overline{G}$, is the inverse limit
\[
\overline{G} = \varprojlim (G/G_i ; \pi_i) = \{(g_k)_{k=0}^{\infty} \in \prod_{k=0}^{\infty} G/G_k \text{ } | \text{ } \pi_n(g_n) = g_{n-1} \text{ for all } n \geq 1\}
\]

\end{definition}

An element $g \in G$ acts on an element $y = (y_i)_{i = 0}^{\infty} \in \overline{G} $ as $g \cdot y = (g \cdot y_i)_{i=0}^{\infty}$.

First we prove that $G$ embeds into $\overline{G}$.

\begin{lemma}\label{embed}
Let $\phi: G \to \overline{G}$ be defined as $g \mapsto (gG_1, gG_2, \ldots)$. Then $\phi$ is an embedding.
\end{lemma}
\begin{proof}
Let $g_1, g_2 \in G$. Suppose
\[
\phi(g_1) = (g_1G_1, g_1G_2, g_1G_3, \ldots) = (g_2G_1, g_2G_2, g_2G_3, \ldots) = \phi(g_2)
\]
So $g_1G_i = g_2G_i$ for all $i$. Therefore $g_1^{-1}g_2 \in G_i$ for all $i$, and so $g_1^{-1}g_2 \in \bigcap G_i = \{0\}$. So $g_1 = g_2$.
\end{proof}

So we have shown that $G$ embeds into $\overline{G}$ in a natural way. We now prove that $(\overline{G},G)$ is minimal.

\begin{lemma}\label{minimal}
The system $(\overline{G},G)$ is minimal.
\end{lemma}
\begin{proof}
Consider the identity element, $e \in \overline{G}$. In particular, $e = (G_1, G_2, G_3, \ldots)$. Let $y = (y_i)_{i=0}^{\infty} \in \overline{G}$. So, for each $n$, we have $y_n = \overline{y_n} G_n$, where $\overline{y_n} \in G$ is a representative of the coset. Note \begin{align*}
\overline{y_n} \cdot e &= \overline{y_n} (G_1, G_2, G_3, \ldots, G_n, \ldots) \\
&= (\overline{y_n}G_1, \overline{y_n}G_2, \overline{y_n}G_3, \ldots, \overline{y_n}G_n, \ldots) \\
&= (\overline{y_1} G_1, \overline{y_2}G_2, \ldots, \overline{y_n} G_n, \ldots) = (y_1, y_2, \ldots, y_n, \ldots)
\end{align*}

So $e \cdot \overline{y_n}$ agrees with $y$ in the first $n$ coordinates. And so we can get arbitrarily close to $y$ as we increase $n$. Hence $e$ has a dense orbit.

Now let $a, b \in \overline{G}$. Note we can find a sequence $b_n$ such that $b_n \cdot e \to ab^{-1}$, since $e$ has a dense orbit. Then $(b_n \cdot e) \cdot b \to a$ so $b_n \cdot b \to a$. Therefore $b$ has dense orbit.
\end{proof}

\begin{definition}[Centralizer]
Let $(X,G)$ be a dynamical system. The \textit{centralizer}, $C(G)$ is defined as
\[
C(G) = \{\varphi \in \text{Homeo}(X) \text{ }|\text{ }g\varphi = \varphi g\text{ for all }g \in G\}
\]
\end{definition}

That is, the centralizer of a system consists of all homeomorphisms of the system which commute with the group action. It can be checked that this is a group under composition.

Next we show that elements of the centralizer of an odometer act as translations of the odometer.

\begin{lemma}\label{translation}
Let $\varphi \in C(\overline{G},G)$. There exists $g_0 \in \overline{G}$ such that $\varphi(x) = x \cdot g_0$ for all $x \in \overline{G}$.
\end{lemma}
\begin{proof}
Let $x \in \overline{G}$. Since the orbit of $e$ is dense, by Lemma \ref{minimal}, there exists a sequence $\{g_n\} \subseteq G$ such that $g_n \cdot e \to x$. Since $\varphi$ is continuous, $\varphi(g_n \cdot e) \to \varphi(x)$. But $\varphi(g_n \cdot e) = g_n \cdot \varphi(e)$ for all $n$. Since $g_n \cdot e \to x$, we have $g_n \to x$. So $\varphi(g_n \cdot e) \to x \cdot \varphi(e)$. Therefore $\varphi(x) = x \cdot \varphi(e)$.
\end{proof}

We are now ready to prove the following Proposition. In the following, $G$ is an abelian group.

\begin{proposition}\label{centralizer}
The centralizer $C(\overline{G},G) = \{\varphi : \overline{G} \to \overline{G}\text{  } |\text{  } \varphi g = g \varphi \text{  } \forall g \in G\}$ of an odometer $\overline{G}$ is isomorphic to $\overline{G}$.
\end{proposition}

\begin{proof}
Define $\psi: C(\overline{G},G) \to \overline{G}$ as $\psi(\varphi) = \varphi(e)$ for all $\varphi \in C(\overline{G},G)$. Let $\varphi_1, \varphi_2 \in C(\overline{G},G)$. Then \begin{align*}
\psi(\varphi_1 \circ \varphi_2) &= \varphi_1 \circ \varphi_2(e)\\
&= \varphi_1(\varphi_2(e))\\
&= \varphi_2(e) \cdot \varphi_1(e)\\
&= \varphi_1(e) \cdot \varphi_2(e)\\
&= \psi(\varphi_1)\psi(\varphi_2)
\end{align*}
So $\psi$ is a homomorphism. Let $y \in \overline{G}$. Let $\varphi_y(x) = x \cdot y$ for all $x \in \overline{G}$. Note, for $g \in \overline{G}$, we have $\varphi_y(gx) = g \varphi_y(x)$ so $\varphi_y \in C(\overline{G},G)$. Also, $\psi(\varphi_y) = y$, so $\psi$ is onto.
Suppose $\psi(\varphi_1) = \psi(\varphi_2)$. Then $\varphi_1(e) = \varphi_2(e)$. So for any $x \in \overline{G}$, $\varphi_1(x) = x \cdot \varphi_1(e) = x \cdot \varphi_2(e) = \varphi_2(x)$. Therefore $\psi$ is an isomorphism.
\end{proof}

We now turn our attention to almost $1-1$ extensions of odometers.

\begin{definition}\label{Toeplitz}
We say $(X,G)$ is an \textit{almost $1-1$ extension} of $(Y,G)$ if there is a factor map $\pi : X \to Y$ such that there is at least one $y \in Y$ so that $\pi^{-1}y$ is singleton. Almost $1-1$ extensions of odometers are also called \textit{Toeplitz Systems}.
\end{definition}

We make use of the following commutative diagram:
\[
\begin{CD}
X @>G>> X\\
@V\pi VV @VV \pi V \\
Y @>>G> Y
\end{CD}
\]

Sometimes the context will deem the action of $G$ on $X$ or $Y$ ambiguous, so we will use $T^gx$ to denote the action of the group element $g \in G$ on $x\in X$ and $S^gy$ to denote the action of $g$ on $y \in Y$. In particular, $\pi \circ T^g = S^g \circ \pi$. If the context is clear, the action of $g$ on a point $x$ will be denoted $g \cdot x$.

If $(X,G)$ is a minimal almost $1-1$ extension of a minimal equicontinuous system, $(Y,G)$, then it is known that $(Y,G)$ is the maximal equicontinuous factor of $(X,G)$ (\cite{auslander1988minimal}). As such, the odometer of which a Toeplitz system $(X,G)$ is an almost $1-1$ extension is its maximal equicontinuous factor.

We will be considering almost $1-1$ extensions of $\mathbb{Z}^d$-odometers. In this context, we will the following proposition.

\begin{proposition}\label{abeliantoeplitz}
The centralizer $C(G)$ of the almost $1-1$ extension of a $\mathbb{Z}^d$-odometer is abelian.
\end{proposition}

To prove Proposition \ref{abeliantoeplitz}, we show that the centralizer of the almost $1-1$ extension of an odometer embeds into the centralizer of its maximal equicontinuous factor, which we have already shown to be isomorphic to the odometer, which is abelian in the case of $G = \mathbb{Z}^d$.

\begin{definition} [\cite{veech1970point}]
Given a dynamical system $(X,G)$ and a metric $d$ compatible with the topology on $X$, two points $x_1$, $x_2 \in X$ are called \textit{proximal} if \[
\inf_{g \in G} d(g \cdot x_1,g \cdot x_2) = 0
\]
\end{definition}

\begin{lemma}\label{proximal}
Let $(X,G)$ be an almost $1-1$ extension of an odometer $(\overline{G},G)$ via the factor map $\pi$. Then points of $X$ are proximal if and only if they are in the same $\pi$ fiber.
\end{lemma}

\begin{proof}
Let $x_1,x_2 \in X$ be in the same $\pi$ fiber, i.e. $\pi (x_1) = \pi (x_2)$. Let $y \in \overline{G}$ be such that $\pi^{-1}y$ is a singleton. Since $(\overline{G},G)$ is minimal, there exists a sequence $\{g_n\}$ such that $\lim_{n \to \infty} S^{g_n} \pi x_1 = y$ and so $\lim_{n \to \infty} S^{g_n} \pi x_2 = y$. Since $X$ is compact, there is a subsequence $\{T^{g_{n_k}}\}$ such that $T^{g_{n_k}}x_1$ converges. Suppose $\lim_{n \to \infty} T^{g_{n_k}}x_1 = z$. Applying $\pi$, we have \begin{align*}
\lim_{n \to \infty} \pi T^{g_{n_k}}x_1 &= \pi z \\
\lim_{n \to \infty} S^{g_{n_k}} \pi x_1 &= \pi z = y
\end{align*}
So we also have $\lim_{n \to \infty} S^{g_{n_k}} \pi x_2 = \pi z = y$.
Since $\pi^{-1}y$ is a singleton, $z$ is unique.
Now, $d(g_n \cdot x_1,g_n \cdot x_2) \leq d(g_n \cdot x_1,z) + d(z,g_n \cdot x_2)$. So,
\begin{align*}
\limsup_{g \in G} d(g_n \cdot x_1,g_n \cdot x_2) &\leq \limsup_{g \in G}( d(g_n \cdot x_1,z) + d(z,g_n \cdot x_2))\\
&\leq \limsup_{g \in G} d(g_n \cdot x_1,z) +\limsup_{g \in G} d(z,g_n \cdot x_2) = 0
\end{align*}
So the points $x_1$ and $x_2$ are proximal.

Now suppose $x_1,x_2 \in X$ are proximal. Assume that $\pi x_1 \neq \pi x_2$, i.e. they are not in the same $\pi$ fiber. Since $x_1, x_2$ are proximal, there is a sequence $\{g_n\} \subseteq G$ such that $\lim_{n \to \infty} T^{g_n} x_1 = \lim_{n \to \infty} T^{g_n} x_2 = z$. Applying $\pi$, we have $\lim_{n \to \infty} \pi T^{g_n} x_1 = \lim_{n \to \infty} \pi T^{g_n} x_2 = \pi z$. So $\lim_{n \to \infty} S^{g_n} \pi x_1 = \lim_{n \to \infty} S^{g_n} \pi x_2$ which implies $\pi x_1, \pi x_2 \in \overline{G}$ are proximal. But $\overline{G}$ has no proximal points, so $\pi x_1 = \pi x_2$.
\end{proof}

Finally, we prove that the centralizer of $X$ embeds into the centralizer of the odometer $Y$.

\begin{proposition}\label{onetoone}
Let $(X,G)$ be an almost $1-1$ extension of a $G-$odometer $(Y,G)$. Every element $\varphi \in C(X,G)$ determines $\psi_{\varphi} \in C(Y,G)$ such that the following diagram commutes:
\[
\begin{CD}
X @>\varphi>> X\\
@V\pi VV @VV \pi V \\
Y @>>\psi_{\varphi}> Y
\end{CD}
\]
Additionally, this relationship is an embedding, i.e. $\psi_{\varphi_1} = \psi_{\varphi_2} \Rightarrow \varphi_1 = \varphi_2$.
\end{proposition}

\begin{proof}
Let $\varphi \in C(X,G)$. Let $x_1, x_2 \in X$ be proximal. So $\pi x_1 = \pi x_2$. Since $x_1$ and $x_2$ are proximal, $\inf_{g \in G} d(g \cdot x_1, g \cdot x_2) = 0$. Thus $\inf_{g \in G} d(\varphi (g \cdot x_1), \varphi (g \cdot x_2)) = 0$ which, by Lemma \ref{proximal}, implies that $\varphi (x_1), \varphi (x_2)$ are proximal. So $\varphi$ preserves the proximal relationship, and so it preserves $\pi$ fibers. Define $\psi_{\varphi}: Y \to Y$ as $\psi_{\varphi} = \pi \circ \varphi \circ \pi^{-1}$. This map is well defined because $\varphi$ preserves the $\pi$ fibers. Suppose $\psi_{\varphi}(y_1) = \psi_{\varphi}(y_2)$ for $y_1, y_2 \in Y$. So $\pi \circ \varphi \circ \pi^{-1} (y_1) = \pi \circ \varphi \circ \pi^{-1} (y_2)$, and so $\varphi \circ \pi^{-1} (y_1)$ and $\varphi \circ \pi^{-1} (y_2)$ are  in the same $\pi$ fibers. Since $\varphi$ preserves the $\pi$ fibers, $\pi^{-1} (y_1)$ and $\pi^{-1} (y_2)$ are in the same $\pi$ fibers, and so it is clear that $y_1 = y_2$. Therefore $\psi_{\varphi}$ is $1-1$. Also, $\psi_{\varphi}$ is continuous, so it is a homeomorphism, i.e. $\psi_{\varphi} \in C(Y,G)$.

Now suppose $\psi_{\varphi_1} = \psi_{\varphi_2}$. Let $y \in Y$ be such that $\pi^{-1} y = \{x\}$ is a singleton. Then $\varphi_1(x) = \pi^{-1}(\psi_{\varphi_1}(y))$ and $\varphi_2(x) = \pi^{-1}(\psi_{\varphi_2}(y))$. Since $\varphi_i$ preserves $\pi$ fibers, for $i \in \{1,2\}$, these are singletons. In particular, $\varphi_1(x) = \varphi_2(x)$. So it is clear then that $g \cdot \varphi_1(x) = g \cdot \varphi_2(x)$ for all $g \in G$, and so $\varphi_1(g \cdot x) = \varphi_2(g \cdot x)$ for all $g \in G$. But every orbit is dense, so $\varphi_1$ and $\varphi_2$ agree on a dense subset of $X$, and hence agree everywhere.
\end{proof}

Finally we prove proposition \ref{abeliantoeplitz}.

\begin{proof}
We have shown in Proposition \ref{onetoone} that $C(X,G)$ embeds into $C(Y,G)$ and by Proposition \ref{centralizer} $C(Y,G)$ is abelian, so $C(X,G)$ is abelian.
\end{proof}

\section{$\mathbb{Z}^d$-Toeplitz Systems}\label{symbolicdynamics}

In this section, we study Toeplitz systems over $\mathbb{Z}^d$ and generalize the construction of Bulatek and Kwiatkowski. In particular, we present a class of Toeplitz systems over $\mathbb{Z}^d$ with a trivial centralizer and positive entropy.

Let $x \in \Sigma^{\mathbb{Z}^d}$. Note that the topological closure of the orbit of $x$, $\overline{O(x)}$ is closed and $T$-invariant. So $(\overline{O(x)}, T)$ is a subshift. This is called the \textit{orbit closure} of $x$.

\begin{definition}
The centralizer of a symbolic dynamical system is called \textit{trivial} if every element $S$ of the centralizer is $T^g$ for some $g \in \mathbb{Z}^d$.
\end{definition}

Let $Z \subseteq \mathbb{Z}^d$ be a finite index subgroup of $\mathbb{Z}^d$ isomorphic to $\mathbb{Z}^d$. For $x \in \Sigma^{\mathbb{Z}^d}$ and $\sigma \in \Sigma$, define
\[
Per(x, Z, \sigma) = \{w \in \mathbb{Z}^d| x(w + z) = \sigma \text{   }\forall z \in Z\}
\]

And,
\[
Per(x, Z) = \bigcup_{\sigma \in \Sigma} Per(x, Z, \sigma)
\]

We say that $x \in \Sigma^{\mathbb{Z}^d}$ is a \textit{Toeplitz array} if for all $v \in \mathbb{Z}^d$, there exists a finite index subgroup $Z \subseteq \mathbb{Z}^d$ isomorphic to $\mathbb{Z}^d$ such that $v \in Per(x, Z)$.

It can be shown that the orbit closure of a Toeplitz Array is an almost one-to-one extension of a $\mathbb{Z}^d$ odometer. For details, the reader is referred to Theorem 7 and Proposition 21 in \cite{cortez2006z}. In fact, almost one-to-one extensions of odometers are exactly those systems which are orbit closures of Toeplitz Arrays. In particular, defining a Toeplitz System as the orbit closure of a Toeplitz Array is equivalent to Definition \ref{Toeplitz}.

We now show how Toeplitz Arrays can be constructed over an alphabet $\Sigma$ borrowing ideas from Downarowicz (\cite{downarowicz2005survey}).

Let $\{p_{t,i}\}_{t=0}^{\infty}, 1 \leq i \leq d$ be $d$ sequences of positive integers such that $p_{0,i} \geq 2$ and $p_{t,i}$ divides $p_{t+1,i}$ for all $0 \leq i \leq d$. Define $\lambda_{t,i} = p_{t+1,i}/p_{t,i}$ and $\lambda_{0,i} = p_{0,i}$ for all $1 \leq i \leq d$ and $t \geq 0$.

An array $\mathbb{Z}^d$ is a point in our system. Any finite rectangular block consisting of letters from our alphabet is called a \textit{finite block}. For a finite block $D$ in $d$ dimensions, we denote the size of $D$ along the $i^{\text{th}}$ dimension as $|D|_i$. We identify the element in the $(i_1,i_2, \ldots, i_d)$ position as $D(i_1,i_2, \ldots, i_d)$ with the standard Cartesian coordinate system, i.e. the left most and bottom-most entry of $D$ is identified with $D(0, 0, \ldots, 0)$.

Specify blocks $A_t$ as follows:

\begin{enumerate}
	\item $|A_t|_i = p_{t,i}$
    \item Some spaces in $A_t$ are filled with elements from $\Sigma$ and others are left unfilled. The unfilled spaces are called \textit{holes}.
    \item The block $A_{t+1}$ is the concatenation of $\lambda_{t+1,i}$ copies of $A_t$ along the $i^{\text{th}}$ dimension for all $1 \leq i \leq d$, where some holes are filled by symbols from $\Sigma$.
    \item For every $(i_1, i_2, \ldots, i_d) \in \mathbb{N}^d$ there exists a $t \geq 0$ such that $A_t(i_1, i_2, \ldots, i_d) \in \Sigma$ and $A_t(p_{t,1}-i_1, p_{t,2}-i_2, \ldots, p_{t,d}-i_d) \in \Sigma$.
\end{enumerate}

We obtain a Toeplitz array $\omega \in \Sigma^{\mathbb{Z}^d}$ by continually repeating the above process. Note that the process described will only tile the first orthant. So to tile the entire space, at each step we shift the origin to be located in the center of our block $A_t$. Continuing this process, we will tile the whole space.

The fourth condition assures that all holes are eventually filled. Note that if after any finite step all holes are filled we will have a periodic array.

Essentially, in this construction we build finite blocks, each of which contains multiple copies of the block built in the previous step. As we copy theses blocks, we fill in some the holes, and leave some them as holes. As we continue this process forever, we will have a Toeplitz Array covering the whole plane.

\begin{example}[One dimensional Toeplitz array] (Due to Downarowicz in \cite{downarowicz2005survey})\label{1d toeplitz element}
\\We will construct a Toeplitz array over $\mathbb{Z}$ from the alphabet $\Sigma = \{0, 1\}$. Let $\{p_t\} = \{2, 4, 8, 16, \ldots\}$ and so $\lambda_t = 2$ for all $t \geq 0$. Let $A_0 = 0 \_$, where the $\_$ symbol indicates a hole. To get $A_1$, we copy $A_0$ twice and fill in some of the holes. Say $A_1 = 0 \underline{1} 0 \_$. The underline indicates a hole that was filled in at that step. In each step we will have two holes. For this construction, at each step we will alternately fill in the first hole with $0$ and $1$. Let the limiting sequence of this process be $\omega$. Continuing, we have
\begin{gather}
\begin{align*}
A_2 &= 010\text{\underline{$0$}}010\_ \\
A_3 &= 0100010\text{\underline{$1$}}0100010\_ \\
A_4 &= 010001010100010\text{\underline{$0$}}010001010100010\_ \\
&\vdots \\
\omega &=0100010101000100010001010100010101000101010001000100010101000100\ldots
\end{align*}
\end{gather}
And so we have a Toeplitz array $\omega$. The orbit closure of this point is a Topelitz system.
\end{example}

\begin{example}[Two dimensional Toeplitz array]\label{twod example}
Again we will use the alphabet $\Sigma = \{0, 1\}$ and we will construct a Toeplitz array over $\mathbb{Z}^2$. Let $\{p_{t, 1}\} = \{p_{t, 2}\}= \{2, 4, 8, 16, \ldots\}$. Then $\lambda_{t, 1} = \lambda_{t,2} = 2$ for all $t \geq 0$.

Let
\[
\begin{tikzpicture}
\draw[step=0.5cm,color=gray] (0,0) grid (1,1);
\node at (0.25,0.25) {$0$};
\node at (0.25,0.75) {$1$};
\node at (0.75,0.75) {$1$};
\node at (-0.5, 0.5) {$A_0 = $};
\filldraw[draw=black,fill=black] (0.5,0) rectangle (1,0.5);
\end{tikzpicture}
\]

\[
\begin{tikzpicture}
\draw[step=0.5cm,color=gray] (0,0) grid (2,2);
\node at (0.25,0.25) {$0$};
\node at (0.25,0.75) {$1$};
\node at (0.25, 1.25) {$0$};
\node at (0.25, 1.75) {$1$};
\node at (0.75, 0.25) {$1$};
\node at (0.75,0.75) {$1$};
\node at (0.75, 1.75) {$1$};
\node at (1.25,0.25) {$0$};
\node at (1.25,0.75) {$1$};
\node at (1.25, 1.25) {$0$};
\node at (1.25, 1.75) {$1$};
\node at (1.75,0.75) {$1$};
\node at (1.75, 1.25) {$0$};
\node at (1.75, 1.75) {$1$};
\draw [step=1.0, very thick] (0,0) grid (2,2);
\node at (-0.5, 1) {$A_1 = $};
\filldraw[draw=black,fill=black] (1.5,0) rectangle (2,0.5);
\filldraw[draw=black,fill=black] (0.5,1) rectangle (1,1.5);
\end{tikzpicture}
\]

\[
\begin{tikzpicture}
\draw[step=0.5cm,color=gray] (0,0) grid (4,4);
\node at (0.25,0.25) {$0$};
\node at (0.25,0.75) {$1$};
\node at (0.25, 1.25) {$0$};
\node at (0.25, 1.75) {$1$};
\node at (0.75, 0.25) {$1$};
\node at (0.75,0.75) {$1$};
\node at (0.75, 1.75) {$1$};
\node at (1.25,0.25) {$0$};
\node at (1.25,0.75) {$1$};
\node at (1.25, 1.25) {$0$};
\node at (1.25, 1.75) {$1$};
\node at (1.75,0.75) {$1$};
\node at (1.75, 1.25) {$0$};
\node at (1.75, 1.75) {$1$};
\node at (0.25,2.25) {$0$};
\node at (0.25,2.75) {$1$};
\node at (0.25, 3.25) {$0$};
\node at (0.25, 3.75) {$1$};
\node at (0.75, 2.25) {$1$};
\node at (0.75,2.75) {$1$};
\node at (0.75, 3.75) {$1$};
\node at (0.75, 1.25) {$0$};
\node at (1.25,2.25) {$0$};
\node at (1.25,2.75) {$1$};
\node at (1.25, 3.25) {$0$};
\node at (1.25, 3.75) {$1$};
\node at (1.75,2.75) {$1$};
\node at (1.75, 3.25) {$0$};
\node at (1.75, 3.75) {$1$};
\node at (1.75, 0.25) {$0$};
\node at (2.25,0.25) {$0$};
\node at (2.25,0.75) {$1$};
\node at (2.25, 1.25) {$0$};
\node at (2.25, 1.75) {$1$};
\node at (2.75, 3.25) {$1$};
\node at (2.75, 0.25) {$1$};
\node at (2.75,0.75) {$1$};
\node at (2.75, 1.75) {$1$};
\node at (3.25,0.25) {$0$};
\node at (3.25,0.75) {$1$};
\node at (3.25, 1.25) {$0$};
\node at (3.25, 1.75) {$1$};
\node at (3.75,0.75) {$1$};
\node at (3.75, 1.25) {$0$};
\node at (3.75, 1.75) {$1$};
\node at (2.25,2.25) {$0$};
\node at (2.25,2.75) {$1$};
\node at (2.25, 3.25) {$0$};
\node at (2.25, 3.75) {$1$};
\node at (2.75, 2.25) {$1$};
\node at (2.75,2.75) {$1$};
\node at (2.75, 3.75) {$1$};
\node at (3.25,2.25) {$0$};
\node at (3.25,2.75) {$1$};
\node at (3.25, 3.25) {$0$};
\node at (3.25, 3.75) {$1$};
\node at (3.75,2.75) {$1$};
\node at (3.75, 3.25) {$0$};
\node at (3.75, 3.75) {$1$};
\node at (3.75, 2.25) {$1$};
\draw [step=2.0, very thick] (0,0) grid (4,4);
\node at (-0.5, 2) {$A_2 = $};
\filldraw[draw=black,fill=black] (1.5,2) rectangle (2,2.5);
\filldraw[draw=black,fill=black] (0.5,3) rectangle (1,3.5);
\filldraw[draw=black,fill=black] (2.5,1) rectangle (3,1.5);
\filldraw[draw=black,fill=black] (3.5,0) rectangle (4,0.5);
\end{tikzpicture}
\]
The black squares indicate where the holes are.
Continuing this process, we will have a coloring of the whole plane, which will be a Toeplitz array, say $\omega$.
\end{example}

We call subblocks of $A_{t+1}$ which coincide with indices of the location of concatenated $A_t$ blocks \textit{t-blocks}. We note that $\omega$ consists of the concatenation of $A_t$ blocks in all directions for any $t$, where all $t$-blocks agree in all locations except for where the holes were. In Example \ref{twod example}, the thick lines in $A_1$ indicate the $0$-blocks, and the thick lines in $A_2$ indicate the $1$-blocks.

Now we introduce a condition on constructing Toeplitz arrays which will give rise to Toeplitz Systems with a trivial centralizer.

\textbf{Condition $(*$)} We say a Toeplitz Array satisfies the condition $(*)$ if:
\begin{itemize}
\item Every $t$-block in $A_{t+1}$ is composed of $A_t$ or $A_t$ with all holes filled
\item The perimeter of $A_{t+1}$ is composed of $t$-blocks which are all filled in
\end{itemize}

Let $e_1, e_2, \ldots, e_d$ be the generators of $\mathbb{Z}^d$. For $1 \leq i \leq d$, let $T_i$ denote a shift by the vector $e_i$. In this context, the shift action on the system can be considered $d$ independent shift actions, i.e. $T^g = T^{(g_1, g_2, \ldots, g_d)} = T_1^{g_1} \times T_2^{g_2} \times \ldots \times T_d^{g_d}$.

\begin{definition}
Given a finite alphabet $\Sigma$, a \textit{patch} is a pair $(P,\mathcal{L})$, where $P \subseteq \mathbb{Z}^d$ and $\mathcal{L} : P \to \Sigma$ is a labeling of $P$. For the purposes of this paper, we will only consider rectangular patches which can be defined by $d$ vectors parallel to the coordinate axes.

Given a patch $(P,\mathcal{L})$, we denote the the coordinate closest to the origin in Cartesian space by $P[0]$. Any other location in the patch is denoted by $P[i]$ where $i \in \mathbb{Z}^d$ is a vector pointing to that location, as referenced from $P[0]$. A square block within $P$ is denoted by $P[i-l, i+k]$ where $k, l \in \mathbb{Z}$ and is the (hyper)cube in $P$ located between $P[i-l\bar{1}]$ and $P[i+k\bar{1}]$, where $\bar{1} = (1, 1, \ldots, 1)$.
\end{definition}

\begin{theorem}
Let $\omega$ be a Toeplitz array satisfying the condition $(*)$. Then the centralizer $C(T)$ of $(\overline{O(\omega)},T)$ is trivial.
\end{theorem}
\begin{proof}
Let $(\overline{G},T_1 \times T_2 \times \ldots \times T_d)$ be the maximal equicontinuous factor of $(\overline{O(\omega)},T)$. Denote by $\pi : (\overline{O(\omega)},T) \to (\overline{G},T_1 \times T_2 \times \ldots \times T_d)$ the almost one-to-one factor map. Let $S \in C(T)$. By Proposition \ref{onetoone}, this determines an element $S' \in C(\overline{G},T_1 \times T_2 \times \ldots \times T_d)$ which acts as a translation by some element $h \in G$, by Lemma \ref{translation}. By a result of Hedlund in \cite{hedlund1969endomorphisms}, we note $S$ is determined by a block code $f$ of window size $k \in \mathbb{N}$. In particular, if $u \in \overline{O(\omega)}$, and $z = S(u)$, then
\[
z[i] = f(u[i-k,i+k]) \text{ for all $i \in \mathbb{Z}^d$}
\]
In particular, the automorphism determines what to put in a specific location by looking at a block around that location in the preimage. By choosing appropriate $j \in \mathbb{Z}^d$, we can define $\tilde{S} = S \circ T^j$ which would require $\tilde{S}$ to only look forward. Specifically, for $u \in \overline{O(\omega)}$, and $z = \tilde{S}(u)$ we have \begin{equation}\label{blockmap}
z[i] = f(u[i,i+k]) \text{ for all $i \in \mathbb{Z}^d$}
\end{equation}
for some $k \in \mathbb{N}$. As such, we can assume $S$ is defined as a block map as in (\ref{blockmap}).

Note that $G$ is a product odometer, so $h = (h_1, h_2, \ldots, h_d)$ where $h_i = \sum_{t=0}^{\infty} h_{t,i}p_{t-1,i}$ for $1 \leq i \leq d$ with $0 \leq h_{t,i} \leq \lambda_{t,i}-1$. Each $h_i$ is an element of the one dimensional odometer occurring in the $i^{th}$ coordinate of $h$. Let $m_{t,i} = \sum_{j=0}^t h_{j,i}p_{j-1,i}$ and $m_t = (m_{t,1}, m_{t,2}, \ldots, m_{t,d}) \in \mathbb{Z}^d$. Let $Q_t \subseteq \Sigma^{\mathbb{Z}^d}$ be the clopen cylinder set with $0$'s located in a $t$-dimensional hybercube about the origin. Then $h \in T_1^{m_{t,1}} T_2^{m_{t,2}} \ldots T_d^{m_{t,d}} Q_t$.

We claim that for all $1 \leq i \leq d$ either $m_{t,i} \leq k$ or $m_{t, i} \geq p_{t,i} - k -1$.

Let $x \in \overline{O(\omega)}$ and $y = S(x)$. Suppose that $x$ has a $(t+1)$-block appearing at a location $x[i]$. Then by the construction of Toeplitz subshifts and almost one-to-one extensions, $y$ necessarily has a $(t+1)$-block at the location $y[i-m_t]$.

Let $A$ denote any $(t+1)$-block of $x$. Note that $A = x[i, i+k]$ for some $i \in \mathbb{Z}^d$ and $k \in \mathbb{N}$. This block looks like $A_{t+1}$, which in turn is the concatenation of $A_t$ blocks. In particular, all $t$-blocks are the same, except they may disagree where the holes are located. Specifically, suppose $C$ is a $t$-block and $C[i]$ is the location of the hole in $C$ that is closest to the bottom left corner. In general, we choose the hole whose location vector $i$ has the minimum length. If there is more than one hole with the same minimum length location vector, then we just choose one at random. Note $C[0, i-1]$ is completely determined, and is the same in those locations as every other $t$-block in $x$. The only place where $t$-blocks may potentially disagree is at the holes.

Let $B$ be the $(t+1)$-block in $y$ starting at location $y[i-m_t]$. Suppose the first hole in $B$ occurs at $B[j]$ for some $j \in \mathbb{Z}^d$. This hole occurs at $A[j+m_t]$ in $A$. In order to determine what is at this location in $B$, $S$ looks at a hypercube of side length $k$ around $A[j]$. In view of Equation (\ref{blockmap}), $B[j]$ is determined by $A[j, j+k]$. We note that if $m_{t,i} > k$ for any $1 \leq i \leq d$, then this window would not overlap the hole at $A[j+m_t]$. And since this hole was the hole closest to the bottom left corner, everything in the window $A[j, j+k]$ is not a hole. And so $B[j]$ is uniquely determined, and is not a hole. Since $A$ was an arbitrary $(t+1)$-block, every $(t+1)$-block will have the symbol $B[j]$ located the relative position $j$. In particular, $A[j+m_t] = B[j]$.

We can continue to the next hole in $A$ on the same horizontal level, and the same argument would show that this hole is completely determined. Continuing this argument for every hole, we see that the entire block is completely filled in, and so then $y$ is periodic, which is a contradiction. In general, in $d$ dimensions, we move along hyperplanes in $d-1$ dimensions which are parallel to the coordinate hyperplanes. We fill in all the holes on a constant hyperplane, and then increase levels by one, until we fill in all the holes.

On the other hand, suppose that $B$ is a $(t+1)$-block in $y$ starting at location $y[i+m_t]$. Note that $S^{-1}$  is also determined by a block map. If $S$ is looking forward, then taking a larger $k$ if needed one can show that $S^{-1}$ is a ``past looking" map determined by $z[i] = g([u[i-k,i]])$. Changing the role of $x$ and $y$ and using $S^{-1}$ for $S$ and using the argument similar to the one above, we can show that $m_t < p_t-k$.

The first case is demonstrated for the two dimensional case in Figure \ref{2dproof}. In this figure, $\dot{A_t}$ indicates $A_t$ blocks with all holes filled and the solid black and red squares indicate a hole in $A$ and $B$, respectively.

\begin{centering}
\begin{figure}[h]
\includegraphics[width=8cm, height = 8cm, keepaspectratio]{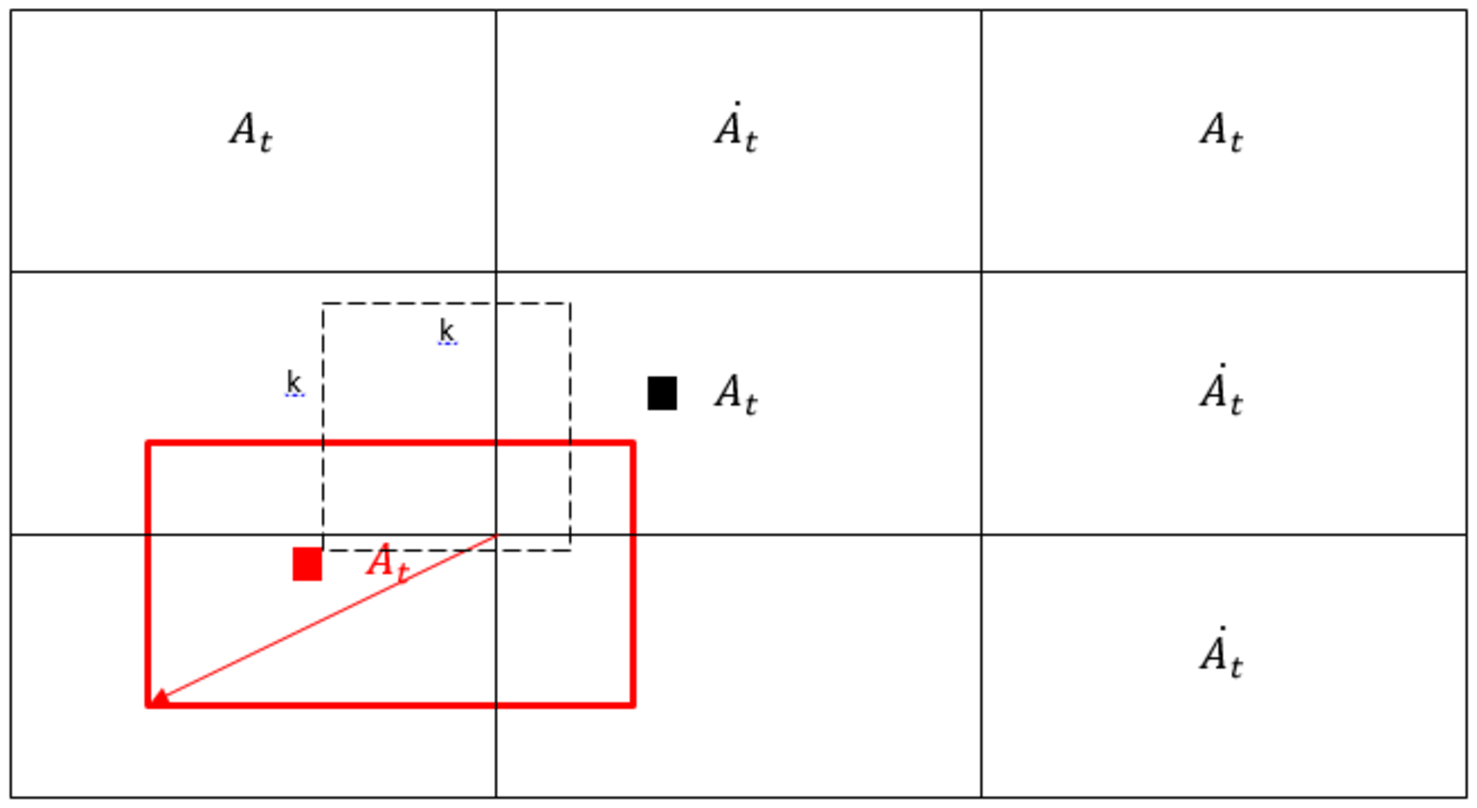}
\caption{ }
\label{2dproof}
\end{figure}
\end{centering}

Now note that $h \in T_1^{m_{t,1}} T_2^{m_{t,2}} \ldots T_d^{m_{t,d}} Q_t$ for all $t$. And since $|m_{t,i}| \leq k \mod p_{t,i}$, we have that $h$ is in the orbit of $0$. In particular,
\[
h \in \bigcap_t T_1^{h_1}T_2^{h_2}\ldots T_d^{h_d} (0) = T^m\{0\}
\]
where $||m||_{\infty} = \max |m_{t,i}| \leq k$.
So, $h = S'(0) = T^h (0)$, i.e. $S'$ and $T$ agree on one point. Furthermore, $S'$ agrees with the action of $h \in \mathbb{Z}^d$ on the entire orbit of $0$, which is dense. Therefore $S'$ is a power of the shift, i.e. $S' = T^h$.

Let $\alpha$ be in the orbit of $\omega$ in $(\overline{O(\omega)},T)$, i.e. $\alpha = T^g \omega$ for some $g \in \mathbb{Z}^d$. Note
\begin{align*}
\pi S(\alpha) &= \pi S(T^g \omega)\\
&= S'\pi(T^g \omega)\\
&= S' T^g(0)\\
&=T^h T^g (0)\\
&= \pi T^h T^g\omega = \pi T^h(\alpha)
\end{align*}
So $S(\alpha)$ and $T^h(\alpha)$ are in the same $\pi$ fiber. Since $\alpha$ is in the orbit of $\omega$, it has a unique preimage under $\pi$. Therefore $S(\alpha) = T^h(\alpha)$. And so $S$ and $T^h$ agree on the entire orbit of $\omega$, which is dense. So $S = T^h$.
\end{proof}

\section{Positive Entropy Toeplitz Subshift} \label{section:positive entropy}

We now construct an explicit example of a two dimensional Toeplitz subshift which has positive entropy. This example is constructed so that it obeys the $(*)$ condition, thus ensuring that it has a trivial centralizer.

Let $h > 0$ and choose $l_0$ such that $\log(l_0-1) \leq h \leq \log(l_0)$. For $i \geq 0$, let $\varepsilon_i > 0$ and $\{\varepsilon_i\}$ be such that $\sum_{i=0}^{\infty} \varepsilon_i < \frac{h}{2}$.

We note that for any $l$ and any $\varepsilon > 0$, there exists $n \in \mathbb{N}$ sufficiently large such that
\begin{equation}\label{inequality}
\frac{\log\left(l^{n^2}\right)}{(n+2)^2} \geq \log(l) - \varepsilon
\end{equation}
since $\left(\frac{n}{n+2}\right)^2 \to 1$.

Let $q_0$ be chosen so that
\[\frac{\log({l_0}^{q_0^2})}{(q_0 + 2)^2} \geq \log(l_0) - \frac{\varepsilon_0}{2}
\]

Also require $q_0^2 \geq l_0$.
Define $l_1 = l_0^{q_0^2}$. We notice that there are $l_0^{q_0^2}$ square blocks of side length $q_0$ over the alphabet $\{0, 1, \ldots, l_0 -1\}$. We enumerate these blocks as $B_i^{(0)}$ for $0 \leq i \leq l_1 -1$. Furthermore, we require that $B_0^{(0)}$ and $B_1^{(0)}$ contain every letter from the alphabet. Let $C_i^{(0)}$ be the square block of side length $q_0 + 2$ with the block $B_i^{(0)}$ surrounded by a $0$ in the top left corner, a $1$ in the bottom right corner, and $0$'s below the main diagonal and $1$'s above it, as in the diagram below. We will denote this as $C_i^{(0)} = 0 B_i^{(0)} 1$ for $0 \leq i \leq l_1-1$.

\[
\begin{tikzpicture}
\draw[step=0.5cm,color=gray] (0,0) grid (0.5,2);
\node at (0.25,0.25) {$0$};
\node at (0.25,0.75) {$\vdots$};
\node at (0.25, 1.25) {$0$};
\node at (0.25, 1.75) {$0$};
\draw[step=0.5cm,color=gray] (0.5, 0) grid (2, 0.5);
\node at (0.75, 0.25) {$\cdots$};
\node at (1.25, 0.25) {$0$};
\node at (1.75, 0.25) {$1$};
\draw[step=0.5cm,color=gray] (1.5, 0.5) grid (2, 2);
\node at (1.75, 0.75) {$1$};
\node at (1.75, 1.25) {$\vdots$};
\node at (1.75, 1.75) {$1$};
\draw[step=0.5cm,color=gray] (0.5, 1.5) grid (1.5, 1.5);
\node at (0.75, 1.75) {$1$};
\node at (1.25, 1.75) {$\cdots$};
\draw[step=1cm,color=white] (0.5, 0.5) grid (1.5, 1.5);
\node at (1, 1) {$B_i^{(0)}$};
\draw[color=gray] (1.5, 0.5) -- (1.5, 2);
\draw[color=gray] (0.5, 1.5) -- (1.5,1.5);
\draw[color=gray] (1, 1.5) -- (1, 2);
\node at (-0.75,1) {$C_i^{(0)} = $};
\end{tikzpicture}
\]

For $k \geq 1$, define $l_k = l_{k-1}^{q_{k-1}^2}$ and let $q_k$ be such that
\begin{equation}\label{one}
\frac{\log({l_k}^{q_k^2})}{(q_k + 2)^2} \geq \log(l_k) - \frac{\varepsilon_k}{2}
\end{equation}
Additionally, require that $q_k^2 \geq l_k$.
Let $B_i^{(k)}$ be all the square blocks of side length $q_k$ over the alphabet $\{0, 1, \ldots, l_k-1\}$ for $0 \leq i \leq l_{k+1} -1$. Require that $B_0^{(k)}$ and $B_1^{(k)}$ contain every letter from the alphabet. Let $C_i^{(k)} = 0 B_i^{(k)} 1$ for $0 \leq i \leq l_{k+1}-1$.
Define $\lambda_k = q_k + 2$ and $p_k = \lambda_1 \lambda_2 \ldots \lambda_k$.

Consider the following operation on finite blocks. Let $\{A_1, A_2, \ldots, A_n\}$ be square blocks of the same side length, $A$ over some alphabet. Let $B$ be a square block whose side length is at least $\sqrt{n}$ over an alphabet containing $\{1, 2, \ldots, n\}$. We define the block
\[
C = \{A_1, A_2, \ldots, A_n\} *  B
\]
as $C[i,j] = A_{B[i.j]}$. In particular, $C$ will be a square block of side length $|B| \cdot A$.

We are constructing a tiling of $\mathbb{Z}^2$ using $k$-blocks as building blocks. Additionally, we must construct these blocks so that they satisfy the $(*)$ condition. As such we define $k$-blocks in the following way:
Let $A_i^{(0)} = C_i^{(0)}$ and
\[
A_i^{(k)} = \{A_0^{(k-1)}, A_1^{(k-1)}, \ldots, A_{l_k-1}^{(k-1)}\} * C_i^{(k)}
\]

We note that since $C_0^{(0)}$ and $C_1^{(0)}$ have every letter of the alphabet $\{0,1, \ldots, l_0-1\}$, the blocks $A_0^{(1)}$ and $A_1^{(0)}$ will have every $0-$block as a subblock. Similarly, $C_0^{(1)}$ and $C_1^{(1)}$ contain every letter in $\{0, 1, \ldots, l_1-1\}$ and so the blocks $A_0^{(2)}$ and $A_1^{(2)}$ will contain every $1-$ block as a subblock. In general, we note that each block $A_i^{(k)}$ for $i = 0, 1$ has every $(k-1)-$block as a subblock.

We let
\[
\begin{tikzpicture}
\draw[step=0.5cm,color=gray] (0,0) grid (0.5,2);
\node at (0.25,0.25) {$0$};
\node at (0.25,0.75) {$\vdots$};
\node at (0.25, 1.25) {$0$};
\node at (0.25, 1.75) {$0$};
\draw[step=0.5cm,color=gray] (0.5, 0) grid (2, 0.5);
\node at (0.75, 0.25) {$\cdots$};
\node at (1.25, 0.25) {$0$};
\node at (1.75, 0.25) {$1$};
\draw[step=0.5cm,color=gray] (1.5, 0.5) grid (2, 2);
\node at (1.75, 0.75) {$1$};
\node at (1.75, 1.25) {$\vdots$};
\node at (1.75, 1.75) {$1$};
\draw[step=0.5cm,color=gray] (0.5, 1.5) grid (1.5, 1.5);
\node at (0.75, 1.75) {$1$};
\node at (1.25, 1.75) {$\cdots$};
\draw[step=1cm,color=white] (0.5, 0.5) grid (1.5, 1.5);
\node at (1, 1) {$\_$};
\draw[color=gray] (1.5, 0.5) -- (1.5, 2);
\draw[color=gray] (0.5, 1.5) -- (1.5,1.5);
\draw[color=gray] (1, 1.5) -- (1, 2);
\node at (-0.5, 1) {$A_0 = $};
\end{tikzpicture}
\]
where the side length of the square box $A_0$ is $q_0 + 2$, and the dash in the center square indicates a square of side length $q_0$ consisting of all holes.

Define
\[
\begin{tikzpicture}
\draw[step=1cm,color=gray] (0,0) grid (5,5);
\node at (0.5,0.5) {$A_0^{(k)}$};
\node at (0.5,1.5) {$A_0^{(k)}$};
\node at (0.5,2.5) {$\vdots$};
\node at (0.5,3.5) {$A_0^{(k)}$};
\node at (0.5, 4.5) {$A_0^{(k)}$};
\node at (1.5, 0.5) {$A_0^{(k)}$};
\node at (1.5, 1.5) {$A_k$};
\node at (1.5, 2.5) {$\vdots$};
\node at (1.5, 3.5) {$A_k$};
\node at (1.5, 4.5) {$A_1^{(k)}$};
\node at (2.5, 0.5) {$\cdots$};
\node at (2.5, 1.5) {$\cdots$};
\node at (2.5, 2.5) {$\ddots$};
\node at (2.5, 3.5) {$\cdots$};
\node at (2.5, 4.5) {$\cdots$};
\node at (3.5, 0.5) {$A_0^{(k)}$};
\node at (3.5, 1.5) {$A_k$};
\node at (3.5, 2.5) {$\vdots$};
\node at (3.5, 3.5) {$A_k$};
\node at (3.5, 4.5) {$A_1^{(k)}$};
\node at (4.5, 0.5) {$A_1^{(k)}$};
\node at (4.5, 1.5) {$A_1^{(k)}$};
\node at (4.5,2.5) {$\vdots$};
\node at (4.5, 3.5) {$A_1^{(k)}$};
\node at (4.5, 4.5) {$A_1^{(k)}$};
\node at (-0.75,2.5) {$A_{k+1} = $};
\fill[gray,opacity=0.3] (0,0) rectangle (1,5);
\fill[gray,opacity=0.3] (1,0) rectangle (5,1);
\fill[gray,opacity=0.3] (4,1) rectangle (5,5);
\fill[gray,opacity=0.3] (1,4) rectangle (4,5);

\end{tikzpicture}
\]

where there is a square block consisting of $q_k^2$ copies of $A_k$ surrounded by $4 q_k + 4$ copies of $A_i^{(k)}$ for $i = 0$ or $1$ on each side. Notice that $A_0^{(k)}$ and $A_1^{(k)}$ have no holes, so all the holes are contained in the middle block of $A_k$ blocks.

Let $\omega$ be the limiting array from the above process. We note here that $\omega$ satisfies the $(*)$ condition.

\begin{proposition}
The Toeplitz system $(\overline{O(\omega)}, T)$ has positive entropy.
\end{proposition}

\begin{proof}
Let $h_{\omega}$ be the entropy of $(\overline{O(\omega)}, T)$ and let $\Theta(n)$ be the number of square blocks of side length $n$ appearing in $\omega$. We note that $h_{\omega} = \lim_{n \to \infty} \frac{\log(\Theta(n))}{n^2} = \lim_{k \to \infty} \frac{\log(\Theta(p_k))}{p_k^2}$, by switching to a subsequence.

There are $l_{k+1}$ many $k-$blocks. We note that every $A_k$ block contains every $(k-1)-$ block as a subblock. This is because the blocks $C_i^{(k)}$ for $i =0$ or $i=1$ contain every letter of the alphabet in them. This means that as we do the shuffling process described above, the blocks $A_i^{(k)}$ for $i = 0$ or $i=1$ contain every single block $A_i^{(k-1)}$ for $0 \leq i \leq l_k-1$. The blocks $A_i^{(k)}$ for $i = 0$ or $i=1$ are exactly those which occur in the $k-$blocks, and so they contain every $(k-1)-$block as a subblock. Furthermore, since $k-$blocks are squares of side length $p_k$, there are at least as many blocks of side length $p_k$ occurring in $\omega$ as there are $k-$blocks. Specifically, square blocks of length $p_k$ can occur at any position within $\omega$, while $k-$blocks only occur at specific positions. Hence we have
\begin{equation}
\Theta(p_k) \geq l_{k+1}
\end{equation}

So we have
\begin{equation}\label{third}
h_{\omega} \geq \limsup_{k \to \infty} \frac{\log(l_{k+1})}{p_k^2}
\end{equation}

By (\ref{one}) we have that \[
\frac{\log(l_{k+1})}{\lambda_k^2} \geq \log(l_k)  - \frac{\epsilon_k}{2}
\]
It then follows, and by (\ref{inequality}), that
\[
\frac{\log(l_{k+1})}{p_k^2} \geq \frac{\lambda_k^2(\log(l_k) - \frac{\varepsilon_k}{2})}{p_k^2} = \frac{\log(l_k) - \frac{\varepsilon_k}{2}}{p_{k-1}^2} \geq \frac{\log(l_k)}{p_{k-1}^2} - \varepsilon_k
\]

Continuing, we have
\[
\frac{\log(l_{k+1})}{p_k^2} \geq h - \sum_{i = 0}^k \varepsilon_i
\]

Taking the limit as $k \to \infty$, from (\ref{third}), we have $h_{\omega} \geq h/2 > 0$.

It is a basic fact that every Toeplitz system is minimal, so this system is minimal. It is either finite or uncountable, and since it has positive entropy, it cannot be finite. So this is an infinite minimal Toeplitz system.

\end{proof}

\newpage
\bibliographystyle{plain}
\bibliography{biblio}

\end{document}